\documentclass[11pt,a4paper, final, twoside]{article}
\usepackage{amsmath}
\usepackage{cite}
\usepackage{fancyhdr}
\usepackage{amsthm}
\usepackage{amsfonts}
\usepackage{amssymb}
\usepackage{amscd}
\usepackage{graphicx}
\usepackage{afterpage}
\usepackage[colorlinks=true, urlcolor=blue,  linkcolor=blue, citecolor=blue]{hyperref}
\newtheorem{theorem}{Theorem}[section]

\newtheorem{corollary}[theorem]{Corollary}

\newtheorem{definition}[theorem]{Definition}
\newtheorem{example}[theorem]{Example}

\newtheorem{lemma}[theorem]{Lemma}

\newtheorem{remark}[theorem]{Remark}

\begin{document}
\begin{center}
{\Large \bf{Common fixed points for $C^{*}$-algebra-valued modular metric spaces via $C_{*}$-class 
functions with application}}
\end{center} \vspace{10mm}
\centerline{Bahman Moeini$^{*,a}$ and Arsalan Hojat Ansari$^{b}$}\ \\
\centerline{\footnotesize  $^{a}$Department of Mathematics, Hidaj Branch, Islamic Azad University, Hidaj, Iran}
\centerline{\footnotesize  $^{b}$Department of Mathematics, Karaj Branch, Islamic Azad University, Karaj, Iran}
\vspace{1mm}
\footnote{$^*$ Corresponding author}
\footnote{E-mail:moeini145523@gmail.com}
\footnote{E-mail:mathanalsisamir4@gmail.com}
\footnote{\emph{}} \afterpage{} \fancyhead{} \fancyfoot{}
\fancyhead[LE, RO]{\bf\thepage} \fancyhead[LO]{\small
$C^{*}$-algebra-valued modular metric spaces with application} \fancyhead[RE]{\small
Moeini and Hojat Ansari}

\begin{abstract}
Based on the concept and properties of $C^{*}$-algebras, the paper introduces a concept of $C_{*}$-class functions. Then by using these functions in $C^{*}$-algebra-valued modular metric spaces of moeini et al. \cite{MACP}, some common fixed point theorems for self-mappings are established. Also, to support of our results an application is provided for existence and uniqueness of solution for a system of integral equations.
\end{abstract}
\section{Introduction}

As is well known, the Banach contraction mapping principle is a very useful, simple and
classical tool in modern analysis, and it has many applications in applied mathematics.
In particular, it is an important tool for solving existence problems in many branches of
mathematics and physics.

In order to generalize this principle, many authors have introduced various types
of contraction inequalities (see \cite{5,7,8,9,13,14}). In 2014,  Ansari \cite{16} introduced the concept of $C$-class functions which cover a large class of contractive conditions. Afterwards, Ansari et al.\cite{ANSEGR} defined and used  concept of complex $C$-class functions involving $C$-class functions in complex valued $G_{b}$-metric spaces to 
obtain some fixed point results. 

One of the main directions in obtaining possible generalizations of fixed point results
is introducing new types of spaces.  In 2010 Chistyakov \cite{CHIS2} defined the notion of modular on an arbitrary set and develop the theory of metric spaces generated by modular such that called the modular metric spaces. Recently, Mongkolkeha et al. \cite{MONG1, MONG2} have introduced some notions and established some fixed point results in modular metric spaces.

 In \cite{ZHLJ2}, Ma et al. introduced the concept of $C^{*}$-algebra-valued metric spaces. The main idea consists in using the set of all positive elements of a unital $C^{*}$-algebra instead of the set of real numbers. 
This line of research was continued in \cite{ZKSR, KPG, ZHLJ, DSTK, ASZ},
where several other fixed point results were obtained in the framework of $C^{*}$-algebra valued
metric, as well as (more general) $C^{*}$-algebra-valued $b$-metric spaces.
 Recently, Moeini et al. \cite{MACP} introduced the concept of $C^{*}$-algebra-valued modular metric spaces which is a generalization of modular metric spaces and next proved some fixed point theorems for self-mappings with contractive conditions on such spaces.
 
In this paper, we introduce a concept of $C_{*}$-class functions on a set of unital $C^{*}$-algebra and  via these 
functions some common fixed point results are proved for self-mappings with contractive conditions in $C^{*}$-algebra-valued modular metric spaces. Also, some examples to elaborate and illustrate of our results are constructed. Finally, as application, existence and uniqueness of solution for a type of system of nonlinear integral
equations is discussed.
\section{Basic notions}
Let $X$ be a non empty set, $\lambda\in (0, \infty)$ and due to the disparity of the arguments, function
$\omega :(0, \infty)\times X\times X \rightarrow [0, \infty]$ will be written as $\omega_{\lambda}(x, y) = \omega(\lambda,x, y)$ for all $\lambda > 0$ and $x,y \in X$.
\begin{definition}\label{d2.1}\cite{CHIS1}
Let $X$ be a non empty set. a function $\omega :(0, \infty)\times X\times X \rightarrow [0, \infty]$ is said
to be a modular metric on $X$ if it satisfies the following three axioms:
\begin{itemize}
\item[$(i)$] given $x,y \in X$,  $\omega_{\lambda}(x, y) = 0$ for all $\lambda > 0$
if and only if  $x=y$;
\item[$(ii)$]  $\omega_{\lambda}(x, y) =\omega_{\lambda}(y, x)$  for all $\lambda > 0$ and $x,y \in X$;
\item[$(iii)$] $\omega_{\lambda+\mu}(x, y) \leq\omega_{\lambda}(x,z)+\omega_{\mu}(z, y)$ for all $\lambda > 0$ and $x,y,z\in X$,
\end{itemize}
\end{definition}
and $(X,\omega)$ is called a modular metric space.\\

Recall that a Banach algebra $\mathbb{A}$ (over the field $\mathbb{C}$ of complex numbers)
is said to be a $C^{*}$-algebra if there is an involution $*$ in $\mathbb{A}$ (i.e., a mapping
$*:\mathbb{A}\rightarrow \mathbb{A}$ satisfying $a^{**}=a$ for each $a\in \mathbb{A}$)
such that, for all $a,b\in \mathbb{A}$ and $\lambda, \mu \in \mathbb{C}$, the following holds:
\begin{itemize}
\item[$(i)$] $(\lambda a+\mu b)^{*}=\bar \lambda a^{*}+\bar \mu b^{*}$;
\item[$(ii)$] $(ab)^{*}=b^{*}a^{*}$;
\item[$(iii)$] $\Vert a^{*}a \Vert=\Vert a \Vert^{2}$.
\end{itemize}
Note that, form $(iii)$, it easy follows that $\Vert a \Vert= \Vert a^{*} \Vert$ for each $a\in \mathbb{A}$.
Moreover, the pair $(\mathbb{A},*)$ is called a unital $*$-algebra if $\mathbb{A}$ contains the identity
element $1_{\mathbb{A}}$. A positive element of $\mathbb{A}$ is an element $a\in \mathbb{A}$
such that $a^{*}=a$ and its spectrum $\sigma(a)\subset \mathbb{R_{+}}$, where $\sigma(a)=\{\lambda\in \mathbb{R}: \lambda 1_{\mathbb{A}}-a \ \text{is noninvertible}\}$.
The set of all positive elements will be denoted by $\mathbb{A_{+}}$. Such elements allow us to
define a partial ordering '$\succeq$' on the elements of $\mathbb{A}$. That is,
\[
b\succeq a \ \ \text{if and only if} \ \ b-a\in\mathbb{A_{+}}.
\]
If $a\in \mathbb{A}$ is positive, then we write $a\succeq \theta$, where $\theta$
is the zero element of $\mathbb{A}$. Each positive element $a$ of a $C^{*}$-algebra $\mathbb{A}$ has a
unique positive square root. From now on, by $\mathbb{A}$ we mean a unital $C^{*}$-algebra with
identity element $1_{\mathbb{A}}$. Further, $\mathbb{A_{+}}=\{a\in \mathbb{A}: a\succeq \theta\}$ and $(a^{*}a)^{\frac{1}{2}}=\vert a\vert$.
\begin{lemma}\label{l2.2} \cite{DRG}
Suppose that $\mathbb{A}$ is a unital $C^{*}$-algebra with a unit $1_{\mathbb{A}}$.
\begin{itemize}
\item[$(1)$] For any $x\in \mathbb{A_{+}}$, we have $x \preceq 1_{\mathbb{A}} \Leftrightarrow
 \Vert x\Vert \leq 1$.
\item[$(2)$] If $a \in \mathbb{A_{+}}$ with $\Vert a\Vert < \frac{1}{2}$,
then $1_{\mathbb{A}}- a $ is invertible and $\Vert a(1_{\mathbb{A}}-a)^{-1}\Vert< 1$.
\item[$(3)$] Suppose that $a, b \in \mathbb{A}$ with $a, b \succeq  \theta$ and
$ab = ba$, then $ab \succeq \theta$.
\item[(4)] By $\mathbb{A'}$ we denote the set $\{a\in \mathbb{A}: ab=ba, \forall b\in \mathbb{A}\}$.
Let $a \in \mathbb{A'}$ if $b, c \in \mathbb {A}$ with
$b \succeq c \succeq \theta$, and $1_{\mathbb{A}}- a \in \mathbb{A'}$ is an
invertible operator, then
\[
(1_{\mathbb{A}}- a)^{-1}b \succeq (1_{\mathbb{A}}- a)^{-1}c.
\]
\end{itemize}
\end{lemma}
Notice that in a $C^{*}$-algebra, if $\theta \preceq a, b$, one cannot conclude that
$\theta \preceq ab$. For example, consider the $C^{*}$-algebra
$\mathbb{M}_{2}(\mathbb{C})$ and set
$a = \left (
\begin{array}{l}
3\ \ \ 2\\
2\ \ \ 3
\end{array}
\right )$,
$b = \left (
\begin{array}{l}
1\ \ \ -2\\
-2\ \ \ 4
\end{array}
\right ),$
then
 $ab = \left (
\begin{array}{l}
-1\ \ \ 2\\
-4\ \ \ 8
\end{array}
\right )$. Clearly $a,b \in \mathbb{M}_{2}(\mathbb{C})_{+}$, while $ab$ is not.\\

\begin{definition}\label{d2.3}\cite{MACP}
Let $X$ be a non empty set. a function $\omega :(0, \infty)\times X\times X \rightarrow \mathbb{A}_{+}$ is said to be a $C^{*}$-algebra-valued modular metric (briefly, $C^{*}$.m.m) on $X$ if it satisfies the following three axioms:
\begin{itemize}
\item[$(i)$] given $x,y \in X$,  $\omega_{\lambda}(x, y) = \theta$ for all $\lambda > 0$
if and only if  $x=y$;
\item[$(ii)$]  $\omega_{\lambda}(x, y) =\omega_{\lambda}(y, x)$  for all $\lambda > 0$ and $x,y \in X$;
\item[$(iii)$] $\omega_{\lambda+\mu}(x, y) \preceq \omega_{\lambda}(x,z)+\omega_{\mu}(z, y)$ for all $\lambda, \mu > 0$ and $x,y,z\in X$.
\end{itemize}
The truple $(X,\mathbb{A}, \omega)$ is called a $C^{*}$.m.m space.
\end{definition}

If instead of $(i)$, we have  the condition\\
$(i')$ $\omega_{\lambda}(x, x) = \theta$ for all $\lambda > 0$ and $x\in X$, then $\omega$ is said to be a $C^{*}$-algebra-valued pseudo modular metric (briefly, $C^{*}$.p.m.m) on $X$
and if $\omega$ satisfies $(i')$, $(iii)$ and\\
$(i'')$ given $x, y \in X$, if there exists a number $\lambda> 0$, possibly depending on $x$ and $y$, such that
$\omega_{\lambda}(x, y) = \theta$, then $x = y$, then $\omega$ is called a $C^{*}$-algebra-valued strict modular metric (briefly, $C^{*}$.s.m.m) on $X$.\\

A $C^{*}$.m.m (or $C^{*}$.p.m.m, $C^{*}$.s.m.m) $\omega$ on $X$ is said to be convex if, instead of $(iii)$, we replace the following condition:
\begin{itemize}\item[$(iv)$]
$\omega_{\lambda+\mu}(x, y) \preceq \frac{\lambda}{\lambda + \mu}\omega_{\lambda}(x,z)+\frac{\mu}{\lambda + \mu}\omega_{\mu}(z, y)$ for all $\lambda, \mu > 0$ and $x,y,z\in X$.
\end{itemize}
Clearly, if $\omega$ is a $C^{*}$.s.m.m, then $\omega$ is a $C^{*}$.m.m, which in turn implies $\omega$ is a $C^{*}$.p.m.m on $X$, and similar implications hold for convex $\omega$. The essential property of a $C^{*}$.m.m $\omega$ on a set $X$ is a following given $x,y\in X$, the function
$0 < \lambda \rightarrow \omega_{\lambda}(x, y) \in \mathbb{A}$ is non increasing
on $(0,\infty)$. In fact, if $0 < \mu < \lambda$, then we have
\begin{equation}\label{eq2.1}
\omega_{\lambda}(x, y)\preceq \omega_{\lambda-\mu}(x, x) +\omega_{\mu}(x, y)=\omega_{\mu}(x, y).
\end{equation}
It follows that at each point $\lambda> 0$ the right limit $\omega_{\lambda+0}(x, y):=\lim_{\varepsilon \rightarrow +0}\omega_{\lambda+\varepsilon}(x, y)$ and the left limit $\omega_{\lambda-0}(x, y):=\lim_{\varepsilon \rightarrow +0}\omega_{\lambda-\varepsilon}(x, y)$ exist in $\mathbb{A}$
and the following two inequalities hold:
\begin{equation}\label{eq2.2}
\omega_{\lambda+0}(x, y)\preceq \omega_{\lambda}(x, y) \preceq\omega_{\lambda-0}(x, y).
\end{equation}
It can be check that if $x_{0}\in X$, the set
\[
X_{\omega}=\{ x\in X: \lim_{\lambda \rightarrow \infty}\omega_{\lambda}(x, x_{0})=\theta\},
\]
is a $C^{*}$-algebra-valued metric space, called a $C^{*}$-algebra-valued modular space, whose
$d_{\omega}^{0}:X_{\omega}\times X_{\omega}\rightarrow \mathbb{A}$ is given by
\[
d^{0}_{\omega}=\inf \{\lambda >0: \Vert \omega_{\lambda}(x, y)\Vert \leq \lambda \}\ \text{ for all}\ x,y \in X_{\omega}.
\]
Moreover, if $\omega$ is convex, the set $X_{\omega}$ is equal to
\[
X_{\omega}^{*}=\{x\in X: \exists \ \lambda=\lambda(x)>0\ \text{such that}\ \Vert \omega_{\lambda}(x, x_{0})\Vert <\infty \},
\]
and $d_{\omega}^{*}:X_{\omega}^{*}\times X_{\omega}^{*}\rightarrow \mathbb{A}$ is given by
\[
d^{*}_{\omega}=\inf \{\lambda >0: \Vert \omega_{\lambda}(x, y)\Vert \leq 1 \}\ \text{ for all}\ x,y \in X_{\omega}^{*}.
\]
It is easy to see that if $X$ is a real linear space, $\rho: X\rightarrow\mathbb{A}$ and
\begin{equation}\label{eq2.3}
\omega_{\lambda}(x, y)=\rho(\frac{x-y}{\lambda})\ \text{ for all} \ \lambda>0 \ \text{and}\ x,y\in X,
\end{equation}
then $\rho$ is  $C^{*}$-algebra valued modular (convex $C^{*}$-algebra-valued modular) on $X$ if and only if $\omega$ is $C^{*}$.m.m (convex  $C^{*}$.m.m, respectively) on $X$.
On the other hand, if $\omega$ satisfy the following two conditions:
\begin{itemize}
\item[$(i)$] $\omega_{\lambda}(\mu x,0)=\omega_{\frac{\lambda}{\mu}}(x, 0)$  for all $\lambda, \mu >0$ and  $x\in X$;
\item[$(ii)$] $\omega_{\lambda}( x+z,y+z)=\omega_{\lambda}(x, y)$ for all $\lambda>0$ and $x,y,z\in X$.
\end{itemize}
If we set $\rho(x)= \omega_{1}( x,0)$ with (\ref{eq2.3}) holds, where  $x\in X$, then
\begin{flushleft}
(a) $X_{\rho}=X_{\omega}$ is a linear subspace of $X$ and the functional $\Vert x\Vert_{\rho}=
d_{\omega}^{0}(x,0), \ x\in X_{\rho}$ is a $F$-norm on $X_{\rho}$;
\end{flushleft}
\begin{flushleft}
(b) If $\omega$ is convex, $X_{\rho}^{*}\equiv X_{\omega}^{*}=X_{\rho}$ is a linear subspace of $X$ and the functional $\Vert x\Vert_{\rho}=
d_{\omega}^{*}(x,0), \ x\in X_{\rho}^{*}$ is a norm on $X_{\rho}^{*}$.
\end{flushleft}
Similar assertions hold if replace $C^{*}$.m.m by $C^{*}$.p.m.m. If $\omega$ is $C^{*}$.m.m in
$X$, we called the set $X_{\omega}$ is $C^{*}$.m.m space. \\

By the idea of property in $C^{*}$-algebra-valued metric spaces and  $C^{*}$-algebra-valued modular spaces, we defined the following:
\begin{definition}\label{d2.4}\cite{MACP}
Let $X_{\omega}$ be a $C^{*}$.m.m space.
\begin{itemize}
\item[(1)] The sequence $(x_{n})_{n\in \mathbb{N}}$ in $X_{\omega}$ is said to be $\omega$-convergent to $x\in X_{\omega}$ with respect to $\mathbb{A}$ if
\begin{center}
$\omega_{\lambda}(x_{n},x)\rightarrow \theta$ as $n\rightarrow \infty$ for all $\lambda>0$.
\end{center}
\item[(2)] The sequence $(x_{n})_{n\in \mathbb{N}}$ in $X_{\omega}$ is said to be $\omega$-Cauchy  with respect to $\mathbb{A}$ if
\begin{center}
$\omega_{\lambda}(x_{m},x_{n})\rightarrow \theta$ as $m,n\rightarrow \infty$ for all $\lambda>0$.
\end{center}
\item[(3)] A subset $C$ of $X_{\omega}$ is said to be $\omega$-closed with respect to $\mathbb{A}$ if the limit of the $\omega$-convergent sequence of $C$ always belong to $C$.
\item[(4)]  $X_{\omega}$ is said to be $\omega$-complete if any $\omega$-Cauchy sequence with respect to $\mathbb{A}$ is $\omega$-convergent.
\item[(5)] A subset $C$ of $X_{\omega}$ is said to be $\omega$-bounded with respect to $\mathbb{A}$ if for all $\lambda>0$
\begin{center}
$\delta_{\omega}(C)=\sup\{\Vert \omega_{\lambda}(x,y)\Vert; \ x,y \in C\}<\infty$.
\end{center}
\end{itemize}
\end{definition}
\begin{definition}\label{d2.5}\cite{MACP}
Let $X_{\omega}$ be a $C^{*}$.m.m space. Let $f, g$ self-mappings of $X_{\omega}$. a point $x$ in
$X_{\omega}$ is called a coincidence point of $f$ and $g$ iff $fx = gx$. We shall call $w = fx = gx$ a point of coincidence of $f$ and $g$.
\end{definition}
\begin{definition}\label{d2.6}\cite{MACP}
Let $X_{\omega}$ be a $C^{*}$.m.m space. Two maps $f$ and $g$ of $X_{\omega}$ are said to be weakly compatible if they commute at coincidence points.
\end{definition}
\begin{definition}\label{d2.7}\cite{MACP}
Let $X_{\omega}$ be a $C^{*}$.m.m space. Two self-mappings $f$ and $g$ of $X_{\omega}$ are occasionally weakly compatible (owc) iff there is a point x in $X_{\omega}$ which is a coincidence point of $f$ and $g$ at which $f$ and $g$ commute.
\end{definition}
\begin{lemma}\label{l2.8}\cite{GJBR}
Let $X_{\omega}$ be a $C^{*}$.m.m space and $f, g$ owc self-mappings of  $X_{\omega}$. If $f$ and $g$ have a unique point of coincidence, $w= fx = gx$, then $w$ is a unique common fixed point of $f$ and $g$.
\end{lemma}
In 2017, Ansari et al. \cite{ANSEGR} introduced the concept of complex $C$-class functions as follows:
\begin{definition}
\label{com C-class}  Suppose $S=\{z\in \mathbb{C}: z\succeq 0\}$, then a continuous function $F:S^{2}\rightarrow \mathbb{C}$ is called a complex $C$-class 
function if for any $s,t\in S$, the following conditions hold:

(1) $F(s,t)\preceq s $;

(2) $F(s,t)=s$ implies that either $s=0$ or $t=0$.
\end{definition}

An extra condition on $F$ that $F(0,0)=0$ could be imposed in some cases if required. 
For examples of these functions see \cite{ANSEGR}.

\section{Main results}
In this section, we introduce a $C_{*}$-class function. The main idea consists in using the set of elements of a unital $C^{*}$-algebra instead of the set of complex numbers.  
\begin{definition}($C_{*}$-class function)
\label{C*-class}  Suppose $\mathbb{A}$ is a unital $C^{*}$-algebra, then a continuous function $F:\mathbb{A}_{+}\times \mathbb{A}_{+}\rightarrow  \mathbb{A}$ is called $C_{*}$-class 
function if for any $A,B\in\mathbb{A}_{+}$, the following conditions hold:

(1) $F(A,B)\preceq A $;

(2) $F(A,B)=A$ implies that either $A=\theta$ or $B=\theta$.
\end{definition}
An extra condition on $F$ that $F(\theta,\theta)=\theta$ could be imposed in some cases if required. The letter $\mathcal{C_{*}}$ will denote the class of all $C_{*}$-class functions.
\begin{remark}
The class $\mathcal{C_{*}}$ includes the set of complex $C$-class functions. It is suffitient to take $\mathbb{A}=\mathbb{C}$ in Definition \ref{C*-class}.
 
\end{remark}
The following examples show that the class $\mathcal{C_{*}}$ is nonempty:
\begin{example}\label{ex3.2}
Let $\mathbb{A} = M_{2}(\mathbb{R})$, of all $2 \times 2$ matrices with the usual operation of addition, scalar multiplication, and matrix multiplication. Define norm on $\mathbb{A}$ by $\Vert A\Vert =\Big(\sum^{2}_{i,j=1}\vert a_{ij}\vert^{2}\Big)^{\frac{1}{2}}$, and $*: \mathbb{A} \rightarrow \mathbb{A}$, given by $A^{*}=A$, for all $A\in \mathbb{A}$,
defines a convolution on $ \mathbb{A}$. Thus $\mathbb{A}$ becomes a $C^{*}$-algebra. For
\[
A = \left (
\begin{array}{l}
a_{11}\ \ \ a_{12}\\
a_{21}\ \ \ a_{22}
\end{array}
\right ),
B = \left (
\begin{array}{l}
b_{11}\ \ \ b_{12}\\
b_{21}\ \ \ b_{22}
\end{array}
\right ) \in \mathbb{A}= M_{2}(\mathbb{R}) ,
\]
we denote $A\preceq B$ if and only if $(a_{ij}-b_{ij})\leq 0$, for all $i,j=1,2$. \\
(1) \ \ Define $F_{*}:\mathbb{A}_{+} \times \mathbb{A}_{+}\rightarrow \mathbb{A}$ by
\[
\ \ F_{*}\Big(\left (
\begin{array}{l}
a_{11}\ \ \ a_{12}\\
a_{21}\ \ \ a_{22}
\end{array}
\right ), \left (
\begin{array}{l}
b_{11}\ \ \ b_{12}\\
b_{21}\ \ \ b_{22}
\end{array}
\right )\ \Big) =\left (
\begin{array}{l}
a_{11}-b_{11}\ \ \ a_{12}-b_{12}\\
a_{21}-b_{21}\ \ \ a_{22}-b_{22}
\end{array}
\right ) 
\]
for all $a_{i,j},b_{i,j}\in \mathbb{R_{+}}$, $(i,j\in \{1,2\})$. Then $F_{*}$ is a $C_{*}$-class function.\newline
(2) \ \ Define $F_{*}:\mathbb{A}_{+} \times \mathbb{A}_{+} \rightarrow\mathbb{A}$ by
\[
\ \ F_{*}\Big(\left (
\begin{array}{l}
a_{11}\ \ \ a_{12}\\
a_{21}\ \ \ a_{22}
\end{array}
\right ), \left (
\begin{array}{l}
b_{11}\ \ \ b_{12}\\
b_{21}\ \ \ b_{22}
\end{array}
\right )\ \Big) =m\left (
\begin{array}{l}
a_{11}\ \ \ a_{12}\\
a_{21}\ \ \ a_{22}
\end{array}
\right ) 
\]
for all $a_{i,j},b_{i,j}\in \mathbb{R_{+}}$, $(i,j\in \{1,2\})$, where, $m\in(0,1)$.
Then $F_{*}$ is a $C_{*}$-class function.
\end{example}
\begin{example}\label{ex4.3}
Let $X=L^{\infty}(E)$ and $H=L^{2}(E)$, where $E$ is a lebesgue measurable set. By $B(H)$ we denote
the set of bounded linear operator on Hilbert space $H$. Clearly, $B(H)$ is a $C^{*}$-algebra with the usual operator norm. \\
Define $F_{*}:B(H)_{+}\times B(H)_{+}\rightarrow B(H)$ by
\[
F_{*}(U,V)=U-\varphi(U),
\]
where $\varphi:B(H)_{+}\rightarrow B(H)_{+}$ is is a continuous function such
that $\varphi(U)=\theta$ if and only if $U=\theta$\ ($\theta=0_{B(H)}$). Then $F_{*}$ is a $C_{*}$-class function.
\end{example}

Let $\Phi _{u}$ denote the class of the functions $\varphi :\mathbb{A}_{+}\rightarrow \mathbb{A}_{+}$ which satisfy the following conditions:

\begin{enumerate}
\item[$(\varphi_{1})$] $\varphi$ is continuous and non-decreasing;

\item[$(\varphi_{2})$] $\varphi (T)\succ \theta , T\succ \theta$ \ and $\varphi (\theta)\succeq \theta$.
\end{enumerate}

Let $\Psi $ be a set of all continuous functions $\psi :\mathbb{A}_{+}
\rightarrow \mathbb{A}_{+}$ satisfying the following conditions:

\begin{itemize}
\item[$(\protect\psi_{1})$] $\psi$ is continuous and non-decreasing;

\item[$(\protect\psi _{2})$] $\psi (T)=\theta$ if and only if $T=\theta$.
\end{itemize}
\begin{definition}
A tripled $(\psi ,\varphi ,F_{*})$ where $\psi \in \Psi ,$ $\varphi \in \Phi
_{u} $ and $F_{*}\in \mathcal{C_{*}}$ is said to be monotone if for any $A,B\in \mathbb{A}_{+}$
\begin{equation*}
A\preceq B\Longrightarrow F_{*}(\psi (A),\varphi (A))\preceq F_{*}(\psi
(B),\varphi (B)).
\end{equation*}
\end{definition}
We now give detailed proofs of main results of this paper.
\begin{theorem}\label{t3.1}
Let $X_{\omega}$ be a $C^{*}$.m.m space and $I, J, R, S, T, U : X_{\omega} \rightarrow X_{\omega}$ be self-mappings of $X_{\omega}$ such that the pairs $(SR, I)$ and $(TU, J)$ are occasionally weakly compatible. Suppose there exist $a,b,c\in \mathbb{A}$ with $0<\Vert a\Vert^{2}+\Vert b\Vert^{2}+\Vert c\Vert^{2}\leq1$ such that the following assertion for all $x,y\in X_{\omega}$ and $\lambda >0$ hold:
\begin{itemize}
\item[$(3.1.1)$] $\psi(\omega_{\lambda}(SRx,TUy))\preceq F_{*}\Big( \psi(M(x,y)),\varphi(M(x,y)\Big)$,
where 
\[
M(x,y)=a^{*}\omega_{\lambda}(Ix,Jy)a+b^{*}\omega_{\lambda}(SRx,Jy)b+c^{*}\omega_{2\lambda}(TUy,Ix)c,
\] 
 $\psi \in \Psi, \varphi\in \Phi_{u}$ and $F_{*}\in \mathcal{C_{*}}$ such that $(\psi, \varphi, F_{*})$ is monotone;
\item[$(3.1.2)$] $\Vert \omega_{\lambda}(SRx,TUy)\Vert< \infty$.
\end{itemize}
Then $SR, TU, I$ and $J$ have a common fixed point in $X_{\omega}$. Furthermore if the pairs $(S,R), (S,I),(R,I),(T,J),(T,U), (U,J)$ are commuting pairs of mappings then $I, J, R, S, T$ and $U$ have a unique
common fixed point in $X_{\omega}$.
\end{theorem}
\begin{proof}
Since the pair $(SR, I)$ and $(TU, J)$ are occasionally weakly compatible then there exists
$u,v\in X_{\omega}: SRu=Iu$ and $TUv=Jv$. Moreover; $SR(Iu) =I(SRu)$ and $TU(Jv) =J(TUv)$.
Now we can assert that $SRu=TUv$. By $(3.1.1)$, we have
\begin{equation}\label{eq3.4}
\psi(\omega_{\lambda}(SRu,TUv))\preceq F_{*}\Big( \psi(M(u,v)),\varphi(M(u,v)\Big),
\end{equation}
where 
\begin{align}
M(u,v)&= a^{*}\omega_{\lambda}(Iu,Jv)a+b^{*}\omega_{\lambda}(SRu,Jv)b+c^{*}\omega_{2\lambda}(TUv,Iu)c\notag   \\
&=a^{*}\omega_{\lambda}(Iu,Jv)a+b^{*}\omega_{\lambda}(Iu,Jv)b+c^{*}\omega_{2\lambda}(Jv,Iu)c.
\label{11}
\end{align}
By definition of $C^{*}$.m.m space and  inequalites (\ref{eq2.1}), (\ref{eq3.4}) and (\ref{11}), we get
\begin{equation}\label{eq3.5}
\begin{array}{rl}
&\psi(\omega_{\lambda}(Iu,Jv))=\psi(\omega_{\lambda}(SRu,TUv))\\
&\preceq F_{*}\Big(\psi(a^{*}\omega_{\lambda}(Iu,Jv)a+b^{*}\omega_{\lambda}(Iu,Jv)b+c^{*}(\omega_{\lambda}(Iu,Iu)+\omega_{\lambda}(Iu,Jv))c),\\
&\varphi(a^{*}\omega_{\lambda}(Iu,Jv)a+b^{*}\omega_{\lambda}(Iu,Jv)b+c^{*}(\omega_{\lambda}(Iu,Iu)+\omega_{\lambda}(Iu,Jv))c)\Big)\\
&=F_{*}\Big(\psi(a^{*}\omega_{\lambda}(Iu,Jv)a+b^{*}\omega_{\lambda}(Iu,Jv)b+c^{*}\omega_{\lambda}(Iu,Jv)c),\\
&\varphi(a^{*}\omega_{\lambda}(Iu,Jv)a+b^{*}\omega_{\lambda}(Iu,Jv)b+c^{*}\omega_{\lambda}(Iu,Jv)c)\Big)\\
&=F_{*}\Big(\psi(a^{*}(\omega_{\lambda}(Iu,Jv))^{\frac{1}{2}}(\omega_{\lambda}(Iu,Jv))^{\frac{1}{2}}a+b^{*}(\omega_{\lambda}(Iu,Jv))^{\frac{1}{2}}(\omega_{\lambda}(Iu,Jv))^{\frac{1}{2}}b\\
&+c^{*}(\omega_{\lambda}(Iu,Jv))^{\frac{1}{2}}(\omega_{\lambda}(Iu,Jv))^{\frac{1}{2}}c),\\
&\varphi(a^{*}(\omega_{\lambda}(Iu,Jv))^{\frac{1}{2}}(\omega_{\lambda}(Iu,Jv))^{\frac{1}{2}}a+b^{*}(\omega_{\lambda}(Iu,Jv))^{\frac{1}{2}}(\omega_{\lambda}(Iu,Jv))^{\frac{1}{2}}b\\
&+c^{*}(\omega_{\lambda}(Iu,Jv))^{\frac{1}{2}}(\omega_{\lambda}(Iu,Jv))^{\frac{1}{2}}c)\Big)\\
&=F_{*}\Big(\psi((a(\omega_{\lambda}(Iu,Jv))^{\frac{1}{2}})^{*}(a(\omega_{\lambda}(Iu,Jv))^{\frac{1}{2}})+(b(\omega_{\lambda}(Iu,Jv))^{\frac{1}{2}})^{*}(b(\omega_{\lambda}(Iu,Jv))^{\frac{1}{2}})\\
&+(c(\omega_{\lambda}(Iu,Jv))^{\frac{1}{2}})^{*}(c(\omega_{\lambda}(Iu,Jv))^{\frac{1}{2}})),\\
&\varphi((a(\omega_{\lambda}(Iu,Jv))^{\frac{1}{2}})^{*}(a(\omega_{\lambda}(Iu,Jv))^{\frac{1}{2}})+(b(\omega_{\lambda}(Iu,Jv))^{\frac{1}{2}})^{*}(b(\omega_{\lambda}(Iu,Jv))^{\frac{1}{2}})\\
&+(c(\omega_{\lambda}(Iu,Jv))^{\frac{1}{2}})^{*}(c(\omega_{\lambda}(Iu,Jv))^{\frac{1}{2}}))\Big)\\
&\preceq F_{*}\Big(\psi(\Vert a(\omega_{\lambda}(Iu,Jv))^{\frac{1}{2}}\Vert^{2}1_{\mathbb{A}}+\Vert b(\omega_{\lambda}(Iu,Jv))^{\frac{1}{2}}\Vert^{2}1_{\mathbb{A}}+\Vert c(\omega_{\lambda}(Iu,Jv))^{\frac{1}{2}}\Vert^{2}1_{\mathbb{A}}),\\
&\varphi(\Vert a(\omega_{\lambda}(Iu,Jv))^{\frac{1}{2}}\Vert^{2}1_{\mathbb{A}}+\Vert b(\omega_{\lambda}(Iu,Jv))^{\frac{1}{2}}\Vert^{2}1_{\mathbb{A}}+\Vert c(\omega_{\lambda}(Iu,Jv))^{\frac{1}{2}}\Vert^{2}1_{\mathbb{A}})\Big)\\
&=F_{*}\Big(\psi(\Vert \omega_{\lambda}(Iu,Jv)\Vert (\Vert a\Vert^{2}+\Vert b\Vert^{2}+\Vert c\Vert^{2})1_{\mathbb{A}}),\\
&\varphi(\Vert \omega_{\lambda}(Iu,Jv)\Vert (\Vert a\Vert^{2}+\Vert b\Vert^{2}+\Vert c\Vert^{2})1_{\mathbb{A}})\Big).
\end{array}
\end{equation}
So,
\begin{equation}\label{eq3.6}
\begin{array}{rl}
&\psi(\Vert \omega_{\lambda}(Iu,Jv) \Vert 1_{\mathbb{A}})\\
&\leq F_{*}\Big(\psi(\Vert \omega_{\lambda}(Iu,Jv)\Vert (\Vert a\Vert^{2}+\Vert b\Vert^{2}+\Vert c\Vert^{2})1_{\mathbb{A}}),\varphi(\Vert \omega_{\lambda}(Iu,Jv)\Vert (\Vert a\Vert^{2}+\Vert b\Vert^{2}+\Vert c\Vert^{2})1_{\mathbb{A}})\Big)\\
&\preceq F_{*}\Big(\psi(\Vert \omega_{\lambda}(Iu,Jv)\Vert 1_{\mathbb{A}}),\varphi(\Vert \omega_{\lambda}(Iu,Jv)\Vert 1_{\mathbb{A}})\Big).
\end{array}
\end{equation}
Thus, $\psi(\Vert \omega_{\lambda}(Iu,Jv) \Vert 1_{\mathbb{A}})=\theta $ or $\varphi(\Vert \omega_{\lambda}(Iu,Jv) \Vert 1_{\mathbb{A}})=\theta $, which means $Iu=Jv$. Hence $SRu=TUv$ and thus
\begin{equation}\label{eq3.7}
SRu=Iu=TUv=Jv.
\end{equation}
Moreover, if there is another point $z$ such that $SRz=Iz$, and using condition $(3.1.1)$
\begin{equation}\label{eq3.8}
\begin{array}{rl}
&\psi(\omega_{\lambda}(SRz,TUv))\preceq F_{*}\Big(\psi(a^{*}\omega_{\lambda}(Iz,Jv)a+b^{*}\omega_{\lambda}(SRz,Jv)b+c^{*}\omega_{2\lambda}(TUv,Iz)c),\\
&\varphi(a^{*}\omega_{\lambda}(Iz,Jv)a+b^{*}\omega_{\lambda}(SRz,Jv)b+c^{*}\omega_{2\lambda}(TUv,Iz)c)\Big) \\
&=F_{*}\Big(\psi(a^{*}\omega_{\lambda}(SRz,TUv)a+b^{*}\omega_{\lambda}(SRz,TUv)b+c^{*}\omega_{2\lambda}(SRz,TUv)c),\\
&\varphi(a^{*}\omega_{\lambda}(SRz,TUv)a+b^{*}\omega_{\lambda}(SRz,TUv)b+c^{*}\omega_{2\lambda}(SRz,TUv)c)\Big).
\end{array}
\end{equation}
By above similar way, we conclude that
\begin{align*}
&\psi(\Vert \omega_{\lambda}(SRz,TUv)\Vert 1_{\mathbb{A}})\\
 &\preceq F_{*}\Big(\psi(\Vert \omega_{\lambda}(SRz,TUv)\Vert (\Vert a\Vert^{2}+\Vert b\Vert^{2}+\Vert c\Vert^{2}))1_{\mathbb{A}},\\
&\varphi(\Vert \omega_{\lambda}(SRz,TUv)\Vert (\Vert a\Vert^{2}+\Vert b\Vert^{2}+\Vert c\Vert^{2})1_{\mathbb{A}})\Big)\\
&\preceq  F_{*}\Big(\psi(\Vert \omega_{\lambda}(SRz,TUv)\Vert 1_{\mathbb{A}},\varphi(\Vert \omega_{\lambda}(SRz,TUv)\Vert 1_{\mathbb{A}})\Big).
\end{align*}
Therefore, $\psi(\Vert \omega_{\lambda}(SRz,TUv)\Vert 1_{\mathbb{A}})=\theta$ or $\varphi(\Vert \omega_{\lambda}(SRz,TUv)\Vert 1_{\mathbb{A}})=\theta$, which means $SRz=TUv$, and so,
\begin{equation}\label{eq3.9}
SRu=Iu=TUv=Jv.
\end{equation}
Thus from equation (\ref{eq3.8}) and (\ref{eq3.9}) it follows that $SRu=SRz$. Hence, $w=SRu=Iu$ for some $w\in X_{\omega}$ is the unique point of coincidence of $SR$ and $I$. Then by Lemma \ref{l2.8}, $w$ is a unique common fixed point of $SR$ and $I$. So, $SRw=Iw=w$.\\
Similarly, there is another common fixed point $w'\in X_{\omega}: TUw'=Jw'=w'$.\\
For the uniqueness, by $(3.1.1)$ we have
\begin{equation}\label{eq3.10}
\begin{array}{rl}
&\psi(\omega_{\lambda}(SRw,TUw'))=\psi(\omega_{\lambda}(w,w') )\\
&\preceq F_{*}\Big(\psi(a^{*}\omega_{\lambda}(Iw,Jw')a+b^{*}\omega_{\lambda}(SRw,Jw')b+c^{*}\omega_{2\lambda}(TUw,Iw')c),\\
&\varphi(a^{*}\omega_{\lambda}(Iw,Jw')a+b^{*}\omega_{\lambda}(SRw,Jw')b+c^{*}\omega_{2\lambda}(TUw,Iw')c)\Big)\\
&=F_{*}\Big(\psi(a^{*}\omega_{\lambda}(w,w')a+b^{*}\omega_{\lambda}(w,w')b+c^{*}\omega_{2\lambda}(w,w')c,\\
&\varphi(a^{*}\omega_{\lambda}(w,w')a+b^{*}\omega_{\lambda}(w,w')b+c^{*}\omega_{2\lambda}(w,w')c)\Big).
\end{array}
\end{equation}
Thus, 
\begin{align*}
&\psi(\Vert \omega_{\lambda}(w,w')\Vert 1_{\mathbb{A}})\\
&\preceq F_{*}\Big(\psi(\Vert \omega_{\lambda}(w,w')\Vert(\Vert a\Vert^{2}+\Vert b\Vert^{2}+\Vert c\Vert^{2})
1_{\mathbb{A}}), \\
&\varphi(\Vert \omega_{\lambda}(w,w')\Vert(\Vert a\Vert^{2}+\Vert b\Vert^{2}+\Vert c\Vert^{2})
1_{\mathbb{A}})\Big)\\
&\preceq F_{*}\Big(\psi(\Vert \omega_{\lambda}(w,w')\Vert
1_{\mathbb{A}}), 
\varphi(\Vert \omega_{\lambda}(w,w')\Vert
1_{\mathbb{A}})\Big).
\end{align*}
So, $\psi(\Vert \omega_{\lambda}(w,w')\Vert 1_{\mathbb{A}})=\theta$ or $\varphi(\Vert \omega_{\lambda}(w,w')\Vert 1_{\mathbb{A}})=\theta$. Hence $w=w'$. Therefore, $w$ is a unique common fixed point of $SR, TU, I$ and $J$. \\
Furthermore, if we take pairs $(S, R),\ (S, I),\ (R, I),\ (T, J),\ (T, U),\ (U, J)$ are commuting pairs then
\[
\begin{array}{rl}
& Sw = S(SRw) = S(RS)w = SR(Sw)\\
&Sw = S(Iw) = S(RS)w = I(Sw)\\
&Rw = R(SRw) = RS(Rw) = SR(Rw)\\
&Rw = R(Iw) = (Rw),
\end{array}
\]
this shows that $Sw$ and $Rw$ is common fixed point of $(SR, I)$ and this gives
$SRw = Sw = Rw = Iw = w$.
Similarly, we have $TUw = Tw = Uw = Jw = w$.
Hence, $w$ is a unique common fixed point of $S,\ R,\ I,\ J,\ T,\ U$.
\end{proof}
\begin{corollary}\label{c3.2}
Let $X_{\omega}$ be a $C^{*}$.m.m space and $I, J, S, T : X_{\omega} \rightarrow X_{\omega}$ be self-mappings of $X_{\omega}$ such that the pairs $(S, I)$ and $(T, J)$ are occasionally weakly compatible. Suppose there exist $a,b,c\in \mathbb{A}$ with $0<\Vert a\Vert^{2}+\Vert b\Vert^{2}+\Vert c\Vert^{2}\leq1$ such that the following assertion for all $x,y\in X_{\omega}$ and $\lambda >0$ hold:
\begin{itemize}
\item[$(3.2.1)$] $\psi(\omega_{\lambda}(Sx,Ty))\preceq F_{*}\Big(\psi(N(x,y)),\varphi(N(x,y))\Big)$ where,
\begin{align*}
N(x,y)=a^{*}\omega_{\lambda}(Ix,Jy)a+b^{*}\omega_{\lambda}(Sx,Jy)b+c^{*}\omega_{2\lambda}(Ty,Ix)c
\end{align*}
$\psi\in \Psi, \ \varphi\in \Phi_{u}$ and $F_{*}\in \mathcal{C_{*}}$ such that $(\psi, \varphi, F_{*})$ is monotone; 
\item[$(3.2.2)$] $\Vert \omega_{\lambda}(Sx,Ty)\Vert< \infty$.
\end{itemize}
Then $S, T, I$ and $J$ have a unique common fixed point in $X_{\omega}$.
\end{corollary}
\begin{proof}
If we put $R=U :=Ix_{\omega}$ where $Ix_{\omega}$ is an identity mapping on $X_{\omega}$, the result follows from Theorem \ref{t3.1}.
\end{proof}
\begin{corollary}\label{c3.3}
Let $X_{\omega}$ be a $C^{*}$.m.m space and $ S, T : X_{\omega} \rightarrow X_{\omega}$ be self-mappings of $X_{\omega}$ such that $S$ and $T$ are occasionally weakly compatible. Suppose there exist $a,b,c\in \mathbb{A}$ with $0<\Vert a\Vert^{2}+\Vert b\Vert^{2}+\Vert c\Vert^{2}\leq1$ such that the following assertion for all $x,y\in X_{\omega}$ and $\lambda >0$ hold:
\begin{itemize}
\item[$(3.3.1)$] $\psi(\omega_{\lambda}(Tx,Ty))\preceq F_{*}\Big(\psi(O(x,y)),\varphi(O(x,y))\Big)$ where,
\begin{align*}
O(x,y)=a^{*}\omega_{\lambda}(Sx,Sy)a+b^{*}\omega_{\lambda}(Tx,Sy)b+c^{*}\omega_{2\lambda}(Ty,Sx)c
\end{align*}
$\psi\in \Psi, \ \varphi\in \Phi_{u}$ and $F_{*}\in \mathcal{C_{*}}$ such that $(\psi, \varphi, F_{*})$ is monotone; 
\item[$(3.3.2)$] $\Vert \omega_{\lambda}(Tx,Ty)\Vert< \infty$.
\end{itemize}
Then $S$ and $T$ have a unique common fixed point in $X_{\omega}$.
\end{corollary}
\begin{proof}
If we put $I=J:=S$, and $S:=T$ in $(3.2.1)$  and $(3.2.2)$ the result follows from Theorem \ref{t3.1}.
\end{proof}
\begin{corollary}\label{c3.4}
Let $X_{\omega}$ be a $C^{*}$.m.m space and $ S, T : X_{\omega} \rightarrow X_{\omega}$ be self-mappings of $X_{\omega}$ such that $S$ and $T$ are occasionally weakly compatible. Suppose there exist $a\in \mathbb{A}$ with $0<\Vert a\Vert\leq1$ such that the following assertion for all $x,y\in X_{\omega}$ and $\lambda >0$ hold:
\begin{itemize}
\item[$(3.4.1)$] $\psi(\omega_{\lambda}(Tx,Ty))\preceq F_{*}\Big(\psi(a^{*}\omega_{\lambda}(Sx,Sy)a), \varphi(a^{*}\omega_{\lambda}(Sx,Sy)a)\Big)$,
where, $\psi\in \Psi, \ \varphi\in \Phi_{u}$ and $F_{*}\in \mathcal{C_{*}}$ such that $(\psi, \varphi, F_{*})$ is monotone; 
\item[$(3.4.2)$] $\Vert \omega_{\lambda}(Tx,Ty)\Vert< \infty$.
\end{itemize}
Then $S$ and $T$ have a unique common fixed point in $X_{\omega}$.
\end{corollary}
\begin{proof}
If we put $b=c:=\theta$,  in $(3.3.1)$  the result follows from Corollary \ref{c3.3}.
\end{proof}
\section{Examples}
In this section we furnish some nontrivial examples in favour of our results.
\begin{example}\label{ex4.1}
Let $X=\mathbb{R}$ and consider, $\mathbb{A} = M_{2}(\mathbb{R})$ as in Example \ref{ex3.2}.\newline 
Define $\omega: (0,\infty)\times X \times X\rightarrow \mathbb{A_{+}}$ by
\[
\omega_{\lambda}(x,y) = \left (
\begin{array}{l}
 \vert \frac{x-y}{\lambda}\vert\ \ \ \ \ \ \ \ 0\\
\ \ 0\ \ \ \ \ \ \ \  \vert \frac{x-y}{\lambda}\vert
\end{array}
\right ),
\]
for all $x,y\in X$ and $\lambda>0$. It is easy to check that $\omega$ satisfies all the conditions of Definition \ref{d2.3}. So, $(X, \mathbb{A}, \omega)$ is a $C^{*}$.m.m space.
\end{example}
\begin{example}\label{ex4.2}
Let $X = \{ \frac{1}{c^{n}}: n=1,2, \cdots \}$ where $0 < c < 1$ and $\mathbb{A} = M_{2}(\mathbb{R})$. Define $\omega: (0,\infty)\times X \times X\rightarrow \mathbb{A_{+}}$ by
\[
\omega_{\lambda}(x,y) = \left (
\begin{array}{l}
 \Vert  \frac{x-y}{\lambda}\Vert\ \ \ \ \ \ \ \ 0\\
\ \  0\ \ \ \ \ \ \ \  \alpha \Vert \frac{x-y}{\lambda}\Vert
\end{array}
\right ),
\]
for all $x,y\in X$, $\alpha\geq 0$ and $\lambda>0$.
Then it is easy to check that $\omega$ is a $C^{*}$.m.m. 
\end{example}
\begin{example}\label{ex4.3}
Let $X=L^{\infty}(E)$ and $H=L^{2}(E)$, where $E$ is a Lebesgue measurable set. By $B(H)$ we denote
the set of bounded linear operator on Hilbert space $H$. Clearly, $B(H)$ is a $C^{*}$-algebra with the usual operator norm. \\
Define $\omega:(0,\infty)\times X\times X\rightarrow B(H)_{+}$ by
\[
\omega_{\lambda}(f,g)=\pi_{\vert \frac{f-g}{\lambda}\vert},\ \ \ (\forall f,g\in X),
\]
where $\pi_{h}:H\rightarrow H$ is the multiplication operator defined by
\[
\pi_{h}(\phi)=h.\phi,
\]
for $\phi \in H$. Then $\omega$ is a $C^{*}$.m.m  and $(X_{\omega},B(H),\omega)$ is a $\omega$-complete $C^{*}$.m.m space. It suffices to verify the completeness of $X_{\omega}$. For this, let $\{f_{n}\}$ be a $\omega$-Cauchy sequence with respect to $B(H)$, that is for an arbitrary
$\varepsilon> 0$, there is $N\in \mathbb{N}$ such that for all $m, n \geq N$,
\begin{center}
$\Vert \omega_{\lambda}(f_{m},f_{n})\Vert =\Vert \pi_{\vert \frac{f_{m}-f_{n}}{\lambda}\vert}\Vert
=\Vert \frac{f_{m}-f_{n}}{\lambda}\Vert_{\infty}\leq \varepsilon$,
\end{center}
so $\{f_{n}\}$ is a Cauchy sequence in Banach space $X$. Hence, there is a function $f \in X$ and $N_{1} \in \mathbb{N}$ such that
\begin{center}
$\Vert \frac{f_{n}- f}{\lambda}\Vert_{\infty}	\leq \varepsilon,\ \ (n \geq N_{1})$.
\end{center}
It implies that
\begin{center}
$\Vert \omega_{\lambda}(f_{n}, f)\Vert= \Vert \pi_{\vert \frac{f_{n}-f}{\lambda}\vert}\Vert=
\Vert \frac{f_{n}- f}{\lambda}\Vert _{\infty} \leq \varepsilon,\ \ (n \geq N_{1})$.
\end{center}
Consequently, the sequence $\{f_{n}\}$ is a $\omega$-convergent sequence in $X_{\omega}$ and so $X_{\omega}$ is a $\omega$-complete $C^{*}$.m.m space.
\end{example}
\begin{example}\label{ex4.4}
Let $(X,\mathbb{A}, \omega)$ is $C^{*}$.m.m space defined as in Example \ref{ex4.1}. Define $S,T,I,J: X_{\omega}\rightarrow X_{\omega}$ by
\[
Sx=Tx=2,\ \ Jx=4-x,\ \
Ix= \left \{
\begin{array}{l}
\frac{2x}{3}\ \ \ \ \ \ \ \ \ if\ x\in (-\infty,2),\\
2\ \ \ \ \ \ \ \ \ \ if\ x=2,\\
0\ \ \ \ \ \ \ \ \ \ if\ x\in (2,\infty).
\end{array}
\right.
\]
Suppose, 
\[
\left \{
\begin{array}{l}
\psi:\mathbb{A_{+}}\rightarrow \mathbb{A_{+}}\\
\psi(A)=2A,
\end{array}
\right.
\ \ \left \{
\begin{array}{l}
\varphi:\mathbb{A_{+}}\rightarrow \mathbb{A_{+}}\\
\varphi(A)=A,
\end{array}
\right.
\ \ \left \{
\begin{array}{l}
F_{*}:\mathbb{A_{+}}\times \mathbb{A_{+}}\rightarrow \mathbb{A}\\
F_{*}(A,B)=A-B.
\end{array}
\right.
\]
Then, $(\psi,\varphi, F_{*})$ is monotone. For all $x,y\in X_{\omega}=\mathbb{R}$ and $\lambda >0$, we have 

\begin{center}
$0=\Big\Vert \left (
\begin{array}{l}
 0\ \ \ \ \ \ 0\\
 0\ \ \ \ \ \ 0
\end{array}
\right )\Big\Vert=\Vert \omega_{\lambda}(Sx,Ty)\Vert< \infty$.
\end{center}
For every $a,b,c\in \mathbb{A}$ with $0<\Vert a\Vert^{2}+\Vert b\Vert^{2}+\Vert c\Vert^{2}\leq1$, we get  
\begin{center}
$\left (
\begin{array}{l}
 0\ \ \ \ \ \ 0\\
 0\ \ \ \ \ \ 0
\end{array}
\right )=\psi(\omega_{\lambda}(Sx,Ty))\preceq F_{*}\Big(\psi(M(x,y)),\varphi(M(x,y))\Big)$,
\end{center}
for all $x,y\in X_{\omega}$ and $\lambda >0$. Also clearly, the pairs $(S, I)$ and $(T, J)$ are occasionally weakly compatible.  So all the conditions of the Corollary \ref{c3.2} are satisfied and $x=2$ is a unique common fixed point of $S,\ T,\ I$ and $J$.
\end{example}
\section{Application}
Remind that if for $\lambda>0$ and $x,y\in L^{\infty}(E)$, define $\omega: (0,\infty)\times  L^{\infty}(E)\times  L^{\infty}(E) \rightarrow B(H)_{+}$ by
\[
\omega_{\lambda}(x,y)=\pi_{\vert \frac{x-y}{\lambda}\vert},
\]
where, $\pi_{h}:H\rightarrow H$ be defined as in Example \ref{ex4.3},
then $( L^{\infty}(E)_{\omega}, B(H), \omega)$ is a $\omega$-complete $C^{*}$.m.m space.\\

Let $E$ be a Lebesgue measurable set, $X=L^{\infty}(E)$ and $H=L^{2}(E)$ be the Hilbert space. Consider the following system of nonlinear integral equations:
\begin{equation}\label{eq5.11}
x(t)=w(t)+k_{i}(t,x(t))+\mu\int_{E}n(t,s)h_{j}(s,x(s))ds,
\end{equation}
for all $t\in E$, where $w\in L^{\infty}(E)_{\omega}$ is known, $k_{i}(t,x(t)),\ n(t,s),\ h_{j}(s,x(s))$,
$i,j=1,2$ and $i\neq j$ are real or complex valued functions that are measurable both in $t$ and $s$ on
$E$ and $\mu$ is real or complex number, and assume the following conditions:\\
\begin{itemize}
\item[$(a)$] $sup_{s\in E}\int_{E}|n(t,s)|dt=M_{1}<+\infty$, 
\item[$(b)$]  $k_{i}(s,x(s))\in L^{\infty}(E)_{\omega}$
for all $x\in L^{\infty}(E)_{\omega},$ and there exists $L_{1}>1$ such that for all $s\in E$,
\[
\frac{|k_{1}(s,x(s))-k_{2}(s,y(s))|}{\sqrt{2}}\geq L_{1}|x(s)-y(s)|\ \ \text{for all}\  x,y\in L^{\infty}(E)_{\omega},
\]
\item[$(c)$] $h_{i}(s,x(s))\in L^{\infty}(E)_{\omega}$ for all $x\in L^{\infty}(E)_{\omega}$, and there exists $L_{2}>0$ such that for all $s\in E$,
\[
|h_{1}(s,x(s))-h_{2}(s,y(s))|\leq L_{2}|x(s)-y(s)|\ \ \text{for all}\  x,y\in L^{\infty}(E)_{\omega},
\]
\item[$(d)$] there exists $x(t)\in L^{\infty}(E)_{\omega}$ such that 
\begin{center}
$x(t)-w(t)-\mu\int_{E}n(t,s)h_{1}(s,x(s))ds=k_{1}(t,x(t))$,
\end{center}
implies
\[
\begin{array}{rl}
&k_{1}(t,x(t))-w(t)-\mu\int_{E}n(t,s)h_{1}(s,k_{1}(s,x(s)))ds\\
&= k_{1}(t,x(t)-w(t)-\mu\int_{E}n(t,s)h_{1}(s,x(s))ds).
\end{array}
\]
\item[$(e)$] there exists $y(t)\in L^{\infty}(E)_{\omega}$ such that 
\begin{center}
$y(t)-w(t)-\mu\int_{E}n(t,s)h_{2}(s,y(s))ds=k_{2}(t,y(t))$,
\end{center}
implies
\[
\begin{array}{rl}
&k_{2}(t,y(t))-w(t)-\mu\int_{E}n(t,s)h_{i}(s,k_{2}(s,y(s)))ds\\
&= k_{2}(t,y(t)-w(t)-\mu\int_{E}n(t,s)h_{2}(s,y(s))ds).
\end{array}
\]
\end{itemize}
\begin{theorem}\label{t5.1}
With the assumptions (\text{a})-(e), the system of nonlinear integral equations (\ref{eq5.11}) has a unique solution $x^{*}$ in $L^{\infty}(E)_{\omega}$ for each real or complex number $\mu$ with $\frac{1+|\mu|L_{2}M_{1}}{L_{1}}\leq1$.
\end{theorem}
\begin{proof}
Define
\[
Sx(t)=x(t)-w(t)-\mu\int_{E}n(t,s)h_{1}(s,x(s))ds,
\]
\[
Tx(t)=x(t)-w(t)-\mu\int_{E}n(t,s)h_{2}(s,x(s))ds,
\]
\[
Ix(t)=k_{1}(t,x(t)),\ Jx(t)=k_{2}(t,x(t)).
\]
Set $a=\sqrt{\frac{1+\vert \mu\vert M_{1}L_{2}}{L_{1}}}.1_{B(H)}$, $b=c=\theta=0_{B(H)}$, then $a\in B(H)_{+}$ and $0 <\Vert a \Vert^{2}+\Vert b \Vert^{2}+\Vert c \Vert^{2}=\frac{1+\vert \mu\vert M_{1}L_{2}}{L_{1}}\leq1$.\newline
Define
\[
\left \{
\begin{array}{l}
\psi:B(H)_{+}\rightarrow B(H)_{+}\\
\psi(B)=\frac{1}{2}B,
\end{array}
\right.
\ \ \left \{
\begin{array}{l}
\varphi:B(H)_{+}\rightarrow B(H)_{+}\\
\varphi(B)=\frac{1}{4}B,
\end{array}
\right.
\ \ \left \{
\begin{array}{l}
F_{*}:B(H)_{+}\times B(H)_{+}\rightarrow B(H)\\
F_{*}(A,B)=\frac{1}{\sqrt{2}}A.
\end{array}
\right.
\]
Then, $(\psi,\varphi, F_{*})$ is monotone.\newline
For any $h\in H$, we have
\[
\begin{array}{rl}
&\Vert \psi(\omega_{\lambda}(Sx,Ty))\Vert =\frac{1}{2}\sup_{\Vert h \Vert=1}(\pi_{\vert \frac{Sx-Ty}{\lambda} \vert}h, h)\\\\
&=\sup_{\Vert h \Vert=1}\int_{E}\Big[ \frac{1}{2\lambda}\Big\vert (x-y)+\mu\int_{E}n(t,s)(h_{2}(s,y(s)-h_{1}(s,x(s))ds \Big\vert\Big]h(t)\overline{h(t)}dt\\\\
&\leq \sup_{\Vert h \Vert=1}\int_{E}\Big[ \frac{1}{2\lambda}\Big\vert (x-y)+\mu\int_{E}n(t,s)(h_{2}(s,y(s)-h_{1}(s,x(s))ds \Big\vert\Big]\vert h(t)\vert^{2}dt\\\\
&\leq \frac{1}{2\lambda}\sup_{\Vert h \Vert=1}\int_{E} \vert h(t)\vert^{2}dt\Big[\Vert x-y\Vert_{\infty}+ \vert \mu \vert M_{1}L_{2}\Vert x-y\Vert_{\infty} \Big]\\\\
&\leq(\frac{1+\vert \mu\vert M_{1}L_{2}}{2\lambda})\Vert x-y\Vert_{\infty}\\\\
&\leq\frac{1}{\sqrt{2}}(\frac{1+\vert \mu\vert M_{1}L_{2}}{2L_{1}})\Vert \frac{k_{1}(t,x(t))-k_{2}(t,y(t))}{\lambda}\Vert_{\infty}\\\\
&=\frac{1}{\sqrt{2}}.\frac{1}{2}(\frac{1+\vert \mu\vert M_{1}L_{2}}{L_{1}})\Vert \omega_{\lambda}(Ix, Jy) \Vert\\\\
&=\frac{1}{\sqrt{2}}.\frac{1}{2}\Vert a\Vert^{2}\Vert \omega_{\lambda}(Ix, Jy) \Vert\\\\
&=\frac{1}{\sqrt{2}}.\frac{1}{2}\Big(\Vert a\Vert^{2}\Vert \omega_{\lambda}(Ix, Jy) \Vert+
\Vert b\Vert^{2}\Vert \omega_{\lambda}(Sx, Jy) \Vert+\Vert c\Vert^{2}\Vert \omega_{2\lambda}(Ty, Ix) \Vert
\Big)\\\\
&=\Vert F_{*}\Big(\psi(N(x,y)), \varphi(N(x,y))\Big)\Vert
\end{array}
\]
Then, 
\begin{center}
$\Vert \psi(\omega_{\lambda}(Sx,Ty))\Vert \leq \Vert F_{*}\Big(\psi(N(x,y)), \varphi(N(x,y))\Big)\Vert $,
\end{center}
for all $x,y\in  L^{\infty}(E)_{\omega}$ and $\lambda>0$. Also by conditions $(d)$ and $(e)$ the pairs $(S,I)$ and $(T,J)$ are occasionally weakly compatible. Therefore, by the Corollary \ref{c3.2}, there exists a unique common fixed point $x^{*}\in L^{\infty}(E)_{\omega}$
such that $x^{*}=Sx^{*}=Tx^{*}=Ix^{*}=Jx^{*}$, which proves the existence of unique solution of (\ref{eq5.11}) in $L^{\infty}(E)_{\omega}$. This completes the proof.
\end{proof}

\end{document}